\documentclass[12pt, reqno]{amsart}
\usepackage[english]{babel}
\usepackage[utf8]{inputenc}
\usepackage[T1]{fontenc}
\usepackage{hyperref}

\usepackage{dirtytalk}
\usepackage{amsmath}%
\usepackage{amsthm}
\usepackage{amssymb}%
\usepackage{mathtools}
\usepackage{setspace} 
\usepackage[a4paper, total={6in, 8in}]{geometry}
\usepackage{xcolor}
\usepackage[normalem]{ulem}
\usepackage{lipsum}  
\newcommand{\C}{\mathbb{C}}
\newcommand{\R}{\mathbb{R}}
\newcommand{\N}{\mathbb{N}}
\newcommand{\Hq}{\mathbb{H}}

\newcommand{\bH}{\mathbb{H}}
\newcommand{\la}{\lambda}
\newcommand{\hcal}{\mathcal{H}}

\newcommand{\bP}{\mathbb{P}}
\newcommand{\e}{\epsilon}
\newcommand{\Ran}{\mbox{Ran}}
\newcommand{\Hi}{\mathcal{H}}
\newcommand{\B}{\mathcal{B}}
\newcommand{\A}{\mathcal{A}}

\newtheorem{theorem}{Theorem}

\newtheorem{lemma}[theorem]{Lemma}
\newtheorem{proposition}[theorem]{Proposition}
 \newtheorem{definition}[theorem]{Definition}

\begin{document}

\thanks{The first and second authors were partially supported by FCT through CAMGSD, project UID/MAT/04459/2019. The third author was partially supported by FCT through CMA-UBI, project UIDB/00212/2020. The fourth author was partially supported by FCT through CIMA, project UIDB/04674/2020.}

\title{On the relation between S-Spectrum and Right Spectrum}

\author[L. Carvalho]{Lu\'{\i}s Carvalho}
\address{Lu\'{\i}s Carvalho, Instituto Universitário de Lisboa (ISCTE-IUL), Business Research Unit (BRU-IUL)\\    Av. das For\c{c}as Armadas\\     1649-026, Lisbon\\   Portugal, and Center for Mathematical Analysis, Geometry,
and Dynamical Systems\\ Mathematics Department,
Instituto Superior T\'ecnico, Universidade de Lisboa\\  Av. Rovisco Pais, 1049-001 Lisboa,  Portugal}
\email{luis.carvalho@iscte-iul.pt}
\author[C. Diogo]{Cristina Diogo}
\address{Cristina Diogo, Instituto Universitário de Lisboa (ISCTE-IUL)\\    Av. das For\c{c}as Armadas\\     1649-026, Lisbon\\   Portugal, and Center for Mathematical Analysis, Geometry,
and Dynamical Systems\\ Mathematics Department,
Instituto Superior T\'ecnico, Universidade de Lisboa\\  Av. Rovisco Pais, 1049-001 Lisboa,  Portugal
}
\email{cristina.diogo@iscte-iul.pt}
\author[S. Mendes]{S\'{e}rgio Mendes}
\address{S\'{e}rgio Mendes, Instituto Universitário de Lisboa (ISCTE-IUL)\\    Av. das For\c{c}as Armadas\\     1649-026, Lisbon\\   Portugal\\ and Centro de Matem\'{a}tica e Aplica\c{c}\~{o}es \\ Universidade da Beira Interior \\ Rua Marqu\^{e}s d'\'{A}vila e Bolama \\ 6201-001, Covilh\~{a}}
\email{sergio.mendes@iscte-iul.pt}
\author[H. Soares]{Helena Soares}
\address{Helena Soares, Instituto Universitário de Lisboa (ISCTE-IUL), Business Research Unit (BRU-IUL)\\    Av. das For\c{c}as Armadas\\     1649-026, Lisbon\\   Portugal\\ and Centro de Investigação em Matemática e Aplicações \\ Universidade de Évora \\  Colégio Luís António Verney, Rua Romão Ramalho, 59, 7000-671 Évora, Portugal}
\email{helena.soares@iscte-iul.pt}

\date{\today}

\begin{abstract}


We use the $\R$-linearity of $I\lambda-T$ to define $\sigma(T)$, the right spectrum of a right $\Hq$-linear operator $T$ in a right quaternionic Hilbert space. We show that $\sigma(T)$  coincides with the $S$-spectrum $\sigma_S(T)$.

\end{abstract}

\maketitle

\section{Introduction}
\onehalfspacing
%

In a complex Hilbert space $\hcal_\C$ the spectrum of a bounded operator $T \in \B(\hcal_\C)$ is the well-known set
\[
\sigma_{\C}(T)=\big\{\lambda \in \C: I\lambda -T \mbox{ is non invertible in } \B(\hcal_\C)\big\}.
\]
When $\Hi$ is a (right) quaternionic Hilbert space, the  spectrum is more elusive, due to the non commutativity of scalar multiplication. 
Nevertheless, quaternionic Hilbert spaces have been used  with the notion of point spectrum, i.e eigenvalues. It is however, clear that this notion is not sound when $\Hi$ is infinite dimensional and therefore, some careful adaptations need to be done.

It is well-known that if the complex Hilbert space $\hcal_\C$ is finite dimensional then the notion of spectrum is the set of eigenvalues.  But when $\hcal_\C$ is infinite dimensional, the spectrum of an operator $T$ is more than just the eigenvalues. More precisely, according to the nature of the failure of the invertibility of $T_\la:=I\lambda-T\in\B(\hcal_\C)$, the spectrum is the union of three disjoint sets:
the point spectrum $\sigma_{\C,p}(T)$, the set of complex numbers $\la$ where the operator $T_\la$ is not injective; 
the residual spectrum $\sigma_{\C,r}(T)$, the set of $\la \in \C\setminus\sigma_{\C,p}(T)$ where the  range of $T_\la$ is not dense in $\hcal_\C$, $\overline{\Ran(T_\la)}\neq \hcal$; 
and the continuous spectrum $\sigma_{\C,c}(T)$, the set of $\la\in\C\setminus(\sigma_{\C,p}(T)\cup\sigma_{\C,r}(T))$ where $T_\la$ is not bounded below. 

To define the spectrum of a quaternionic right linear operator $T \in \B(\hcal)$ we cannot simply
emulate what happens in the complex case. In fact, the operator of right multiplication by a quaternion, $I\lambda: \hcal \to \hcal$ defined by $I\lambda(x)\coloneqq x\lambda$, is not right $\Hq$-linear. Thus,  $T_\la$ is also not right $\Hq$-linear, i.e. $T_\la$ is not in $\B(\Hi)$, but on the contrary, right multiplication by a quaternion is clearly $\R$-linear and so is $T_{\lambda}$.
%
%
Therefore, it is natural to define a notion of spectrum based on the invertibility of the $\R$-linear operator $T_\la$. In other words, we consider the larger space of bounded $\R$-linear operators on $\Hi$, denoted by $\B_\R(\Hi)$. Accordingly, we define a modified right spectrum to be the subset of $\Hq$
\begin{equation}\label{R-spectrum}
\sigma(T)=\big\{\lambda \in \Hq: I\lambda -T \mbox{ is non invertible in } \B_\R(\Hi)\big\},
\end{equation}
which resembles the usual notion of spectrum in the complex setting. The spectrum defined this way can be decomposed in the usual family of disjoint sets, 
  \[
  \sigma(T)=\sigma_p(T)\cup\sigma_r(T)\cup\sigma_c(T),
  \]
where $\sigma_p(T)$, $\sigma_r(T)$ and $\sigma_c(T)$ are respectively the point, residual and continuous spectra, defined again by 
     \begin{align*}
        \sigma_p(T)=&\{\la \in \bH: T_\la(x)=0,\; \text{for some}\; x \in \hcal\setminus \{0\}\},\\
        \sigma_r(T)=&\{\la\in\Hq\setminus \sigma_p(T): \overline{\Ran(T_\la)}\neq \hcal \},\\
        \sigma_c(T)=&\{\la \in \Hq\setminus\big(\sigma_c(T)\cup \sigma_p(T)\big): T_\la \mbox{ is not bounded below}\}.
    \end{align*}



On the other hand, the quaternionic functional calculus has seen a recent major breakthrough and the fundamental stepping-stone of that leap is the introduction of the operator $\Delta_\la(T)=T^2-2Re(\la)T+|\la|^2I$ in $\B(\hcal)$, for any $\la \in \Hq$. The non-invertibility of $\Delta_\la(T)$, for a given $T \in \B(\hcal)$, defines a new notion of spectrum, the $S$-spectrum,
\begin{equation*}\label{S-spectrum}
\sigma_S(T)=\big\{\lambda \in \Hq: \Delta_\la(T) \mbox{ is non invertible in } \B(\Hi)\big\}.
\end{equation*}

Again, the $S$-spectrum can be split  into three disjoint sets: the point $S$-spectrum $\sigma_{S,p}(T)$, the residual $S$-spectrum $\sigma_{S,r}(T)$, 
and the continuous $S$-spectrum $\sigma_{S,c}(T)$, (see \cite{CGK} and references therein).

This similarity raises the question whether these sets are equal to the ones defined for the right spectrum (\ref{R-spectrum}). The purpose of this paper is to answer affirmatively to this question. We prove that the notions of right spectrum and $S$-spectrum coincide. Actually, we show that each component in the partition of the spectrum coincides with the corresponding component in the partition of the $S$-spectrum, that is 
\[
\sigma_{p}(T)=\sigma_{S,p}(T), \quad \sigma_{c}(T)=\sigma_{S,c}(T), \quad \sigma_{r}(T)=\sigma_{S,r}(T). 
\]

Before delving into the gory details, two comments are worth making. First, our result provides a handy way to calculate the spectrum but the introduction of the operator $\Delta_\la$ was crucial for further developing quaternionic operator theory. In fact, many achievements were unthinkable without the discovery of the S-spectrum. To name a few, a generalization of the Riesz-Dunford functional calculus for holomorphic functions to quaternionic linear operators \cite{CSS}, the continuous functional calculus for normal operators on a quaternionic Hilbert space \cite{GMP} and spectral theorems for unitary \cite{ACKS} and for unbounded normal quaternionic linear operators \cite{ACK}, among others.
The second observation is that the equality of the two spectra notions in this paper is natural and in a sense easy, but to the best of our knowledge, has passed unnoticed in the literature, except the equality of the right and S-spectrum in the finite dimensional case.  There is also one article that mentions the equality of the right spectrum and the S-spectrum in infinite dimension, but again only proves the finite dimensional case, establishing the equality of the right spectrum and the point S-spectrum (see \cite[Theorem 2.5]{CS}).



For convenience of the reader, we recall some basic definitions and results. The division ring of real quaternions $\Hq$ is an algebra over $\R$ with basis $\{1, i, j, k\}$ and  product given by $i^2=j^2=k^2=ijk=-1$. The pure quaternions are denoted by $\bP=\mathrm{span}_{\R}\,\{i,j,k\}$. The real and imaginary parts of a quaternion $q=a_0+a_1i+a_2j+a_3k\in\Hq$ are denoted by $Re(q)=a_0$ and $Im(q)=a_1i+a_2j+a_3k\in\bP$, respectively. The conjugate of $q$ is given by $q^*=Re(q)-Im(q)$ and its norm is $|q|=\sqrt{qq^*}$. Two quaternions $q_1,q_2\in\Hq$ are called similar, and we write $q_1\sim q_2$, if there exists $s \in \Hq$ with $|s|=1$ such that $s^{*}q_2 s=q_1$. Similarity is an equivalence relation and the class of $q$ is denoted by $[q]$. A necessary and sufficient condition for the similarity of $q_1$ and $q_2$ is that $Re(q_1)=Re(q_2) \textrm{ and }|Im(q_1)|=|Im(q_2)|$.

A right $\Hq$-module $\Hi$ equipped with an inner product $\langle.,.\rangle$ is called a right quaternionic Hilbert space if it is complete
with respect to the norm $\|x\|=\sqrt{\langle x,x\rangle}$, $x\in \Hi$.

We can look at $\Hi$ as a vector space over $\Hq$ or $\R$. This allows us to introduce two notions of linear operators over the two fields.
A right $\bH$-linear operator is a map $T: \hcal \to \hcal$ such that 
    \[
    T(u\alpha+v\beta)=T(u)\alpha+T(v)\beta,\; \text{for any} \;u, v \in \hcal \; \text{and} \; \alpha, \beta \in \bH.
    \]
Analogously, if the above equality holds for any $\alpha, \beta \in \R$, we say that $T$ is an $\R$-linear operator. Furthermore, recall that an operator $T: \hcal \to \hcal$ is bounded if there exists $K\geq 0$ such that $\|Tx\|\leq K\|x\|$, $x\in \Hi$, and that the norm of $T$ is defined by
$\|T\|=\mathrm{sup}\,\{\|Tx\|:\|x\|=1\}$.

Accordingly, we denote by $\B(\hcal)$ the set of all bounded right $\bH$-linear operators on $\hcal$ and by $\B_\R(\hcal)$ the set of all bounded $\R$-linear operators on $\hcal$.
Since a $\Hq$-linear map is a $\R$-linear map, we have
$\B(\hcal) \subseteq  \B_\R(\hcal)$.


Given $T\in\mathcal{B}(\Hi)$ and $\la\in\Hq$, we define the operator $\Delta_\la(T):\Hi\longrightarrow \Hi$ by $\Delta_\la(T)=T^2-2Re(\la)T+|\la|^2 I$.
Clearly, $\Delta_\la(T)$ is a bounded right $\Hq$-linear operator. The S-spectrum of $T$, $\sigma_S(T)$, is defined by
\begin{eqnarray*}
\sigma_S(T)   &=& \big\{\la \in \bH: \Delta_\la(T) \mbox{ is not invertible in } \B(\hcal)\big\}.
\end{eqnarray*}
The S-spectrum $\sigma_S(T)$ is a compact nonempty subset of $\Hq$ and it is always contained in the closed ball of radius $\|T\|$ around origin $\overline{B(0, \|T\|)}$. One can show that $\la\in \sigma_S(T)$ is equivalent to $[\la]\subseteq \sigma_S(T)$ (\cite[Theorem 9.2.2, Theorem 9.2.3]{CGK}).




Now, recall that a  bounded linear operator $T$ is invertible  in $\B(\Hi)$ if, and only if, $T$ has a dense range and it is bounded
from below, the latter meaning that there exists $k > 0$ such that $\|Tx\| \geq k\|x\|$ for every $x \in \Hi$ (see \cite[page 196]{CGK}). 
A classical partition of the spectrum into three disjoint parts, according to the nature of the failure of $\Delta_\la(T)$ to be invertible is the following:
\[\sigma_S(T)=\sigma_{S,p}(T)\cup\sigma_{S,c}(T)\cup\sigma_{S,r}(T),\]
where
\begin{align*}
        \sigma_{S,p}(T)=&\{\la \in \bH: \Delta_\la(T)(x)=0, \; \text{for some}\; x \in \hcal\setminus \{0\}\},\\
        \sigma_{S,r}(T)=&\{\la \in \Hq\setminus \sigma_{S,p}(T): \overline{\Ran( \Delta_\la(T)}\neq \hcal \},\\
        \sigma_{S,c}(T)=&\{\la \in \Hq\setminus\big(\sigma_{S,p}(T)\cup \sigma_{S,r}(T)\big): \Delta_\la(T) \mbox{ is not bounded below}\}.
    \end{align*}

\medskip
Consider the operator $T_\la\coloneqq I\lambda -T: \hcal \to \hcal$, defined before as $T_\la(x)=x\la-Tx$. Although $T_\la$ is not right $\Hq$-linear, it is $\R$-linear operator. Therefore, it makes sense to talk about invertibility of $T_\la$ in $\B_{\R}(\Hi)$ and it is natural to define the spectrum of $T\in \B(\Hi)$ as follows.

\begin{definition}
  Let $ T\in \B(\hcal)$. Then the right spectrum of $T$ is the set
   \[
\sigma(T)  = \big\{\la \in \bH: T_\la \mbox{ is not invertible in } \B_{\R}(\hcal)\big\}.
\] 
\end{definition}

\medskip

Similarly, $\sigma(T)$ splits into a disjoint union of three parts
  \[
  \sigma(T)=\sigma_p(T)\cup\sigma_c(T)\cup\sigma_r(T),
  \]
where 
the point spectrum, the continuous spectrum and residual spectrum of $T$ are, respectively, 
     \begin{align*}
        \sigma_p(T)=&\{\la \in \bH: T_\la(x)=0,\; \text{for some}\; x \in \hcal\setminus \{0\}\},\\
        \sigma_r(T)=&\{\la\in\Hq\setminus \sigma_p(T): \overline{\Ran(T_\la)}\neq \hcal \},\\
        \sigma_c(T)=&\{\la \in \Hq\setminus\big(\sigma_r(T)\cup \sigma_p(T)\big): T_\la \mbox{ is not bounded below}\}.
    \end{align*}

\bigskip

Recall that a set $\A\subset \Hq$ is circular if $\la \in \A$, then $[\la] \subset \A$.
It is known that $\sigma_S(T)$ is circular as well as $\sigma_{S,p}, \sigma_{S,r}, \sigma_{S,c}$ (see \cite[page 33]{GMP}). The same is true for  $\sigma(T)$, $\sigma_p(T)$, $\sigma_r(T)$ and $\sigma_c(T)$ as shown in the following lemma. Before,  observe that when $\la =q \mu q^*$, for some unitary quaternion $q$, we have $\big( T_\la x\big)q= T_\mu (xq)$. In fact,
  \begin{align}
        T_\la x=& x\la -Tx=x(q\mu q^*)- T\big(x(q q^*)\big)  \nonumber
        \\
        =& (xq\mu)q^*-T(xq)q^*=\Big((xq)\mu-T(xq)\Big)q^*  \nonumber
        \\
        =&\Big(T_\mu (xq) \Big) q^*. \label{equality}
    \end{align}
 Note that we cannot go further than this since $T_\mu$ is not right $\Hq-$linear, only $T$ is.
     
\begin{lemma} \label{circular}
Let $T \in \B(\hcal)$. The sets $\sigma_p(T)$, $\sigma_r(T)$ and $\sigma_c(T)$ are circular. In particular, $\sigma(T)$ is circular.
\end{lemma}
\begin{proof}
 We will prove this result by showing that when $\la \in \A$, then any element of the form $\mu=q \la q^*$, with unitary $q \in \bH$, 
is also in $\A$, and therefore $[\la] \subseteq \A$, where $\A$ is one of the above spectra.

Let $\lambda \in \sigma_p(T)$. If $x \in \hcal\setminus\{0\}$ is such that $T_\la (x)=0$, taking $y=xq$ and using (\ref{equality}) we have $T_\mu(y)=0$. Then  $\mu \in \sigma_p(T)$ and therefore $\sigma_p(T)$ is circular.

We now prove that  $\sigma_r(T)$ is circular. Let $\lambda \in \sigma_r(T)$. Note that $\lambda \notin \sigma_p(T)$, $\mu=q \la q^*$ and $\sigma_p(T)$ is circular, hence $\mu \notin \sigma_p(T)$.
Since $\bH q^*=\bH$, using \eqref{equality} and taking $y=xq$ we have 
\begin{align*}
\overline{\Ran(T_\mu)} =& \overline{\{T_\mu(y): y \in \hcal \}}= \overline{\{T_\mu(xq) : xq \in \hcal\}}
    \\
    =& \overline{\{T_\la(x) q: xq \in \hcal\}}= \overline{\{T_\la(x) q: x \in \hcal q^*\}}
    \\
    =& \overline{\{T_\la(x) q: x \in \hcal\}}= \overline{\{T_\la(x): x \in \hcal\}q}
    \\
   =& \overline{\Ran(T_\la)} q 
\end{align*}
Therefore, $\overline{\Ran(T_\la)} \neq \hcal$ implies that $\overline{\Ran(T_\la)}q =\overline{\Ran(T_\mu)} \neq \hcal$. So  $\sigma_r(T)$ is circular.

It remains to see that $\sigma_c(T)$ is circular. Let $\lambda \in \sigma_c(T)$. Then $\lambda \notin \sigma_p(T) \cup \sigma_r(T)$, $\mu=q \la q^*$ and $\sigma_p(T)$ and $\sigma_r(T)$ are both circular, so $\mu \notin \sigma_p(T) \cup \sigma_r(T)$. Since $T_\lambda$ is not bounded below then there is a sequence $x_n \in \hcal$ such that $\|T_\lambda x_n\| \to 0$. Take $y_n=x_n q$, with  $q \in \bH$ a unitary quaternion, such that $\la= q \mu q^*$, then, again by (\ref{equality}), we have  $\|T_\mu(y_n)\|= \|\big(T_\la(x_n) \big)q\| \to 0$. Then $T_\mu$ is not bounded below and so $\sigma_c(T)$ is circular.

\end{proof}

Before we prove our main theorem we note that the operators $\Delta_\la(T)$ and $T_\la$ are related by composition $\Delta_\la(T)=T_\la\cdot T_{\la^*}$. In fact, for all $x\in\Hi$, we have
\begin{align*}
    T_\la\cdot T_{\la^*}(x)=&(I\la-T)\cdot(I\la^*-T)(x)=(I\la-T)(x\la^*-Tx)
      \\
    =& (x\la^*-Tx)\la-T(x\la^*-Tx)   =x|\la|^2 -(Tx)\la - (Tx)\la^*+T^2(x)
    \\
    =&T^2(x)-(Tx) (\la+\la^*)+|\la|^2x=\big(T^2-2Re(\la)T+|\la|^2I\big)x
    \\
    =&\Delta_\la(T)(x).
\end{align*}
Likewise we can prove that $T_{\la^*}\cdot T_\la=\Delta_\la(T)$. Summing up we have:

\begin{proposition}\label{prop_delta=TT*}
     Let $T\in\B(\hcal)$. Then 
    \[ \Delta_\la(T)= T_\la\cdot T_{\la^*}= T_{\la^*}\cdot T_\la. \]
\end{proposition}

From this result  it is easy to see that when $T_\la$ is not invertible in $\B_{\R}(\Hi)$, $\Delta_\la(T)$ is not invertible in $\B_\R(\Hi)$ and therefore not invertible in $\B(\Hi)$, which means $\sigma(T)\subset \sigma_S(T)$. We will prove that the converse also holds. This will need some extra work since proposition \ref{prop_delta=TT*} only implies a slightly weaker result,  if $T_\la$ is invertible in $\B_{\R}(\Hi)$, $\Delta_\la(T)$ is invertible in $\B_{\R}(\Hi)$. We can find not only that $\sigma_S(T)=\sigma(T)$, but stronger than that, that the point spectrum of $T$ is the  point $S$-spectrum of $T$, the continuous spectrum is the continuous $S$-spectrum; and the residual spectrum is the  residual $S$-spectrum.
\begin{theorem}\label{equal_spectra}
      Let $T \in \B(\hcal)$. We have the following equalities 
      \[
       \sigma_{p}(T)=\sigma_{S,p}(T), \quad \sigma_{r}(T)=\sigma_{S,r}(T) \quad \mbox{ and } \quad \sigma_{c}(T)=\sigma_{S,c}(T).
      \]
\begin{proof}
  
    If $\la \in \sigma_{S,p}(T)$, then $\Delta_\la(T)x=0$ for some $x \in \hcal\setminus \{0\}$.  Taking $y=T_{\la^*}(x)$ and using proposition \ref{prop_delta=TT*}, we have  $T_\la y=0$. Either $\la \in \sigma_p(T)$, and we are done, or $T_{\la^*}(x)=y=0$, in which case $\la^* \in \sigma_p(T)$. Since $\sigma_p(T)$ is circular, $\lambda \in \sigma_p(T)$. We conclude that $\sigma_{S,p}(T) \subseteq \sigma_p(T)$. For the converse inclusion, if $T_\la(x)=0$ for some $x \in \hcal\setminus \{0\}$, using proposition \ref{prop_delta=TT*}, it follows that $\Delta_\la(T)x=0$ for some $x \in \hcal\setminus \{0\}$, that is, $\lambda\in\sigma_{S,p}(T)$.
    

  Let us now prove that $\sigma_{r}(T)=\sigma_{S,r}(T)$. Since $\overline{\Ran( T_\la \cdot T_{\la^*})} \subseteq \overline{\Ran(T_\la)} \neq \hcal$ we have $\sigma_{r}(T)\subseteq\sigma_{S,r}(T)$. To prove the converse inclusion, we will use the contra-positive: $\overline{\Ran(T_\la)}= \hcal$  implies $\overline{\Ran(\Delta_\la (T)} )= \hcal$.
  From (\ref{equality}) we have $T_\la(x)q=T_{\la^*}(xq)$, where $\la^*=q^*\la q$, with $q$ unitary. Since  $T_{\la^*}(\cdot \; q)\in \B_{\R}(\Hi)$ with $x\mapsto T_{\la^*}(x q)$, we have that 
  \[
\hcal=\hcal q= \overline{\Ran(T_\la)} \;q= \overline{\Ran(T_\la)\;q}= \overline{\Ran(T_{\la^*}\big(\cdot \;q)\big)}=\overline{\Ran(T_{\la^*})}.
\]
It follows that  $\overline{\Ran(T_\la)}=\Hi$ implies  $\overline{\Ran(T_{\la^*})}=\hcal$ (and vice-versa). So assume that  $\overline{\Ran(T_\la)}= \overline{\Ran(T_{\la^*})}=\hcal$.
To prove that $\overline{\Ran(\Delta_\la (T)} )= \hcal$ we will find that, for any $x \in \hcal$, there is a sequence $x_n \in \hcal$ such that $\Delta_\la (T)(x_n)=T_{\la^*} \cdot T_{\la}(x_n)  \overset{n}{\longrightarrow}  x$. Since the range of $T_{\la^*}$ is dense there is a sequence $y_n \in \hcal$ such that $T_{\la^*}(y_n) \overset{n}{\longrightarrow} x$; and since $\Ran (T_{\la})$ is also dense, for each $n$ there is a sequence $y_{n,k}$ such that $T_{\la}(y_{n,k}) \overset{k}{\longrightarrow} y_n$.  
For any $\e>0$, let $N \in \N$ be such that $\|T_{\la^*}(y_n)-x\|<\e$, when $n\geq N$. For any of these $n$ pick $k(n) \in \N$ satisfying $\|T_{\la}(y_{n,k(n)})- y_n\|<\e$. Then,
\begin{align*}
    \|T_{\la^*}\cdot T_{\la}(y_{n,k(n)})- x\|&\leq \|T_{\la^*}\cdot T_{\la}(y_{n,k(n)})- T_{\la^*}(y_n)\|+\|T_{\la^*}(y_n)-x\|\\
    &\leq \|T_{\la^*}\|\| T_{\la}(y_{n,k(n)})-y_n\|+\|T_{\la^*}(y_n)-x\|\\
    & \leq (\|T_{\la^*}\|+1)\e.
\end{align*}
Hence we have a sequence $x_n =y_{n,k(n)} $ such that $\Delta_\la (T)(x_n) \overset{n}{\longrightarrow}  x$, thus $x \in \overline{\Ran(\Delta_\la (T) )}$. 


Finally, we will see that $\sigma_{c}(T)=\sigma_{S,c}(T)$.
     Assume $\la \in \sigma_c(T)$, then $T_\la$ is not bounded below and there is a sequence of unitary vectors $x_n \in\hcal$ such that $T_\la (x_n) \to 0$. By continuity of $T_{\la^*}$, we have $\Delta_\la(T)(x_n)=T_{\la^*}\cdot T_\la (x_n) \to 0$, i.e., $\la \in \sigma_{S,c}(T)$. 
   On the other hand, if $\la \in \sigma_{S,c}(T)$, there is a sequence of unitary vectors $x_n$ such that $\Delta_\la(T)(x_n)=T_{\la}\cdot T_{\la^*} (x_n) \to 0$. Then either  $\lim\inf \|T_{\la^*} (x_n)\| \to 0$, in which case lemma \ref{circular} implies that $\la \in\sigma_c(T)$, or  $\ell \coloneqq \lim\inf \|T_{\la^*} (x_n)\|>0$. We can take a subsequence $x_{n_k}$ of $x_n$ where $\|T_{\la^*} (x_{n_k})\|\geq \ell$.  For simplicity we denote such subsequence by $x_n$. Let $y_n=T_{\la^*} (x_n)/\|T_{\la^*} (x_n)\|$. Clearly,
   \[
   \|T_\la y_n\|=  \frac{\|T_\la \cdot T_{\la^*} (x_n)\|}{\|T_{\la^*} (x_n)\|}=  \frac{\|\Delta_\la(T)(x_n)\|}{\|T_{\la^*} (x_n)\|} \leq \frac{ \|\Delta_\la(T)(x_n)\|}{\ell}\to 0,
   \]
   and so $\la \in \sigma_c(T)$.

\end{proof}
      
\end{theorem}

The main result of the paper is now a direct consequence of theorem \ref{equal_spectra}.

\begin{theorem}
     Let $T \in \B(\hcal)$, then $\sigma(T)=\sigma_S(T)$.
\end{theorem}

\end{document}